\documentclass[a4paper,11pt,onecolumn,twoside]{article}
\usepackage{tikz}
\usetikzlibrary{arrows.meta}
\usepackage{fancyhdr}
\usepackage{hyperref}
\usepackage{bbm}
\usepackage[utf8]{inputenc}
\usepackage{amsmath,amsfonts,amssymb}
\usepackage{graphicx}
\usepackage{mathptmx}
\DeclareUnicodeCharacter{FB01}{fi}
\usepackage{amsthm}
\usepackage{booktabs}
\usepackage{amssymb}
\usepackage[mathscr]{eucal}
\usepackage[labelfont=bf]{caption}
\usepackage{indentfirst}
\usepackage{caption}
\usepackage{enumitem}
\usepackage{subfigure}
\usepackage{authblk}
\usepackage{natbib}
\usepackage[all]{xy}
\usepackage{geometry}
\usepackage{mathrsfs}
\usepackage{newtxmath}
\usepackage{hyperref}
\numberwithin{equation}{section}

\newtheorem{thm}{Theorem}[section]

\newtheorem{con}[thm]{Conjecture}
\newtheorem{definition}[thm]{Definition}

\newtheorem{lemma}[thm]{Lemma}

\newtheorem{prop}[thm]{Proposition}

\newtheorem{rmk}[thm]{Remark}

\newcommand{\GG}{{^{G}_{G}\mathcal{YD}^{\omega}}}

\newcommand{\BGGA}{{^{\mathbb{G}}_{\mathbb{G}}\mathcal{YD}^{\pi^*(\omega_{\underline{a}})}}}
\newcommand{\BGGJ}{{^{\mathbb{G}}_{\mathbb{G}}\mathcal{YD}^{\partial(J_{\underline{a}})}}}
\newcommand{\BGGo}{{^{\mathbb{G}}_{\mathbb{G}}\mathcal{YD}}}
\newcommand{\GGV}{{^{G_V}_{G_V}\mathcal{YD}^{\omega\mid_{G_V}}}}
\newcommand{\GGVV}{{^{\mathbb{G}_{\widetilde{V}}}_{\mathbb{G}_{\widetilde{V}}}\mathcal{YD}}}

\newcommand{\GGAA}{{^{G}_{G}\mathcal{YD}^{\omega_{\underline{a}}}}}

\newcommand{\HH}{{^{H}_{H}\mathcal{YD}}}

\newcommand{\GGB}{{^{\mathbb{G}}_{\mathbb{G}}\mathcal{YD}^{\pi^*\omega}}}

\newcommand{\bVG}{\textbf{Vec}_G^{\omega}}

\newcommand{\sg}{\mathbbm{g}}
\newcommand{\bg}{\mathbb{G}}

\newcommand{\Propcite}[2]{[\citealp{#1}, Proposition~#2]}

\title{\textbf{ Cohomology of Pointed Finite Tensor Categories  }}
\author{
	BOWEN LI AND GONGXIANG LIU}

\date{}

\setlist{nolistsep}
\begin{document}\large
	\maketitle
	
	\setlength{\oddsidemargin}{ -1cm}
	\setlength{\evensidemargin}{\oddsidemargin}
	\setlength{\textwidth}{15.50cm}
	\vspace{-.8cm}
	
	\setcounter{page}{1}
	
	\setlength{\oddsidemargin}{-.6cm}  
	\setlength{\evensidemargin}{\oddsidemargin}
	\setlength{\textwidth}{17.00cm}
		\thispagestyle{fancy}
	\fancyhf{}
	\fancyfoot[R]{\thepage}
	\fancyfoot[L]{Bowen Li, Gongxiang Liu:  School of Mathematics, Nanjing University, Nanjing 210093, P. R. China. \\ B.Li, e-mail: DZ21210002@smail.nju.edu.cn, G.Liu, e-mail: gxliu@nju.edu.cn.\\
		Mathematics Subject Classiﬁcation: 18G15; 16T05.}
	\fancyhead{} 
	\renewcommand{\footrulewidth}{1pt}
	\renewcommand{\headrulewidth}{0pt}
	\begin{abstract}
		We consider the ﬁnite generation property for cohomology algebra of pointed finite tensor categories via de-equivariantization and exact sequence of  finite tensor categories. As a result, we prove that all coradically graded pointed finite tensor categories over abelian groups have finitely generated cohomology. 
	\end{abstract}\par 
\textbf{Keywords}: Coquasi-Hopf algebra; Tensor category;  Finitely generated cohomology
\section{Introduction}
Let $\mathbbm{k}$ be an algebraically closed field of characteristic zero. Given a finite tensor category $\mathcal{C}$ over $\mathbbm{k}$, we may define the cohomology of $\mathcal{C}$ as follows:
$$\operatorname{H}^{\bullet}(\mathcal{C},V)=\operatorname{Ext}_{\mathcal{C}}^{\bullet}(\textbf{1},V)$$
where $V$ is an object in $\mathcal{C}$ and $\textbf{1}$ is the unit object in $\mathcal{C}$. We mainly focus on  cohomological ﬁniteness property: 
\begin{definition}
	A finite tensor category $\mathcal{C}$ over $\mathbbm{k}$ is said to have finitely generated cohomology (\rm{\textbf{FGC}} for brevity), if $\operatorname{H}^{\bullet}(\mathcal{C},\mathbf{1})$ is a ﬁnitely generated algebra and $\operatorname{H}^{\bullet}(\mathcal{C},V)$ is a ﬁnitely generated $\operatorname{H}^{\bullet}(\mathcal{C},\mathbf{1})$ -module for each $V$ in $\mathcal{C}$.
\end{definition}
In \cite{ftc}, Etingof and Ostrik proposed the well-known conjecture:
\begin{con}
	Any finite tensor category over $\mathbbm{k}$ satisfies \rm{\textbf{FGC}}.
	\end{con}
For over twenty years, many beautiful results related to this conjecture have been obtained, especially from the perspective of Hopf algebras. For a finite-dimensional Hopf algebra $H$, then the category $\operatorname{Rep}(H)$ of finite-dimensional representations is a finite tensor category. By abuse of terminology, we say $H$ satisfies \rm{\textbf{FGC}}, if $\operatorname{Rep}(H)$ does. Ginzburg and Kumar in \cite{Fgcsqg} showed that all small quantum groups satisfy \rm{\textbf{FGC}}. Moreover, they computed their cohomology rings explicitly. In \cite{MJPS}, Mastnak, Pevtsova, Schauenberg and Witherspoon proved that any finite-dimensional pointed Hopf algebra over an abelian group (i.e., whose group-like elements form an abelian group) satisfies \rm{\textbf{FGC}}, subject to mild restrictions on the group’s order. This result was later extended to all finite-dimensional pointed Hopf algebras over abelian groups in \cite{AAPS}.  Related results for fields of positive characteristic can be found in  \cite{NWW19} and \cite{FN18}. 
\par On the other hand, some experts study this question from the perspective of category theory. In \cite{coho}, Negron and Plavnik characterized when the dual category and the center of a finite tensor category $\mathcal{C}$ satisfy \rm{\textbf{FGC}}, assuming $\mathcal{C}$
 itself satisfies \rm{\textbf{FGC}}. The condition of finitely generated cohomology is crucial in the theory of support varieties for finite tensor categories; further details may be found in \cite{Supp}  and \cite{2024supportvarietiesfinitetensor}. \par 
We mainly focus on the coradically graded pointed finite tensor category $\mathcal{C}$ over abelian group. Recall that a finite tensor category is pointed if every simple object is invertible. In \cite{AC}, the authors defined a finite tensor category as coradically graded if it is equivalent to the category of finite-dimensional comodules over a coradically graded coquasi-Hopf algebra. Then by the well-known Tannakian reconstruction theorem, there is a finite-dimensional coradically graded coquasi-Hopf algebra $M$ over abelian group, such that $\mathcal{C} \cong \operatorname{Comod}(M)$.  Thus the classification  of coradically graded pointed finite tensor categories over abelian groups(up to tensor equivalence) is equivalent to  the classification of finite-dimensional coradically graded coquasi-Hopf algebras over abelian groups(up to gauge equivalence). \par 
The classification of such coquasi-Hopf algebras was completed in \cite{QQG} and \cite{huang2024classification}.
Based on the type of the associated twisted Nichols algebra, the classification divides into two cases: diagonal type and non-diagonal type.
We address these cases separately to establish our main result on finitely generated cohomology:
\begin{thm}\label{B-thm1.3}
Let $\mathcal{C}$ be a coradically graded pointed finite tensor category  over abelian group, then $\mathcal{C}$ satisfies \rm{\textbf{FGC}}.
\end{thm}
For convenience, we  call a coradically graded coquasi-Hopf algebra of (non-)diagonal type if its corresponding twisted Nichols algebra is of (non-)diagonal type over the group of  group-like elements.\par  
We first consider  coquasi-Hopf algebra $M$ of diagonal type. The classification process is summarized in  figures $(4.3)$ and $(4.17)$  in \cite{QQG}. A key observation is that the figures exhibits a profound connection with the de-equivariantization of finite tensor categories, linking $\operatorname{Comod}(M)$ to the de-equivariantization of $\operatorname{Comod}(H)$, where $H$ is  a finite-dimensional coradically graded Hopf algebra of diagonal type. Using results from \cite{coho} and \cite{AAPS}, we prove this case. It  is worth emphasizing that Angiono and Galindo also proved similar conclusions on de-equivariantization of pointed finite tensor categories in \cite{AC}, but our approach differs by employing a more algebraic perspective.\par 
For a pointed coquasi-Hopf algebra $M$ of non-diagonal type, a distinct method is necessary because its comodule category does not arise as a de-equivariantization of a pointed Hopf algebra comodule category. Instead, we analyze its intrinsic structure. By \cite{huang2024classification}, $M \cong \mathcal{B}(V) \# \mathbbm{k}G$ where  $\mathcal{B}(V)$ is a twisted Nichols algebra of non-diagonal type, $G$ is an abelian group and $\omega$ is a non-abelian cocycle. However, we may relate $M$ to a pointed Hopf algebra of diagonal type via  an exact sequence of tensor categories. Through further calculations and transformations, we demonstrate that $\operatorname{Comod}(M)$ satisfies \rm{\textbf{FGC}} as well. thereby completing the proof of Theorem \ref{B-thm1.3}.
\par 
Section 2 introduces definitions, notations, and foundational results on coquasi-Hopf algebras and tensor categories. Section 3 treats the diagonal type case, while Section 4 addresses the non-diagonal type case and presents the proof of Theorem \ref{B-thm1.3}.

\section{Background on coquasi-Hopf algebra and cohomology } 
This section reviews foundational definitions and results concerning coquasi-Hopf algebras and the cohomology of finite tensor categories.
\subsection{Classification result of finite-dimensional coquasi-Hopf algebras}
We briefly summarize the classification of finite-dimensional coradically graded pointed coquasi-Hopf algebras over finite abelian groups, as established in  \cite{huang2024classification}.\par
 By definition, coquasi-Hopf algebras are exactly the dual of Drinfeld’s quasi-Hopf algebras \cite{Drinfeld}. These are coassociative coalgebras that need not be associative. Analogous to the Hopf algebra case, the classification relies critically on: 
  twisted Yetter-Drinfeld module  categories $\GG$ and twisted Nichols algebras, where $G$ is a finite abelian group and $\omega$ is a normalized $3$-cocycle over $G$.   For formal definitions and examples, see [\citealp{QQG} Section 2].

\par The classification of finite-dimensional coradically graded  coquasi-Hopf algebra over abelian groups hinges on the structure of the coradical $(\mathbbm{k}G,\omega)$, where $\omega$ is a $3$-cocycle on $G$. 
A complete parametrization of 3-cohomology classes for finite abelian groups was achieved in \cite{QQG} as follows. 
Let $G$ be a finite abelian group and so there is no harm to assume that $G=Z_{m_1}\times Z_{m_2}\cdots \times Z_{m_n}$ with $m_i \mid m_{i+1}$ for $1 \leq i \leq n-1$.
Denote $\mathscr{A}$ the the set of all $\mathbb{N}$-sequences:
$$\underline{a}=(a_1,a_2,...,a_l,...,a_n,a_{12},a_{13},...,a_{st},...,a_{n-1,n},a_{123},...,a_{rst},...a_{n-2,n-1,n})$$
such that $0 \leq a_l <m_l$, $0\leq a_{st} <(m_s,m_t)$, $0\leq a_{rst} <(m_r,m_s,m_t)$, with indices ordered lexicographically. Let $g_i$ be the generator of $Z_{m_i}$, $1 \leq i \leq n$. For each $\underline{a} \in \mathscr{A}$, define
\begin{align}\label{B-2.1}
	& \omega\left( g_1^{i_1} \cdots g_n^{i_n}, g_1^{j_1} \cdots g_n^{j_n}, g_1^{k_1} \cdots g_n^{k_n}\right)  = \prod_{l=1}^n \zeta_{m_l}^{a_l i_l\left[\frac{j_l+k_l}{m_l}\right]} \prod_{1 \leq s<t \leq n} \zeta_{m_s}^{a_{s t} k_s\left[\frac{i_t+j_t}{m_t}\right]} \prod_{1 \leq r<s<t \leq n} \zeta_{\left(m_r, m_s, m_t\right)}^{a_{r s t} k_r j_s i_t}
\end{align}
Here $\zeta_m$ denotes a primitive $m$-th root of unity. We further define $\mathscr{A}'$ as the subset of $\mathscr{A}$ satisfying $a_{rst}=0$ for all $1\leq r <s<t \leq n$ and $\mathscr{A}''$ as the subset of $\mathscr{A}$ satisfying $a_{i}=0$, $a_{st}=0$ for all $1\leq i \leq n$, $1\leq s<t \leq n$.
\begin{lemma}\textup{\Propcite{QQG}{3.8}}
	$	\left\lbrace \omega_{\underline{a}}| \underline{a}\in \mathscr{A}\right\rbrace $ forms a complete set of representatives of the normalized $3$-cocycles over $G$ up to $3$-cohomology.
\end{lemma}
Among all the $3$-cocycles over abelian groups, the class of abelian cocycles is particularly significant.  These are intimately connected to the property that a twisted Yetter-Drinfeld module $V$ over $\GG$ is of diagonal type, a crucial aspect of our classification. For details, see [\citealp{QQG}, Section $3$] for details.
\begin{lemma}\textup{[\citealp{QQG}, Corollary 3.13, Proposition 3.14]}\label{B-{lem 5.6}}
Let $G$ be a finite abelian group and $\omega$  a normalized $3$-cocycle over $G$ as in (\ref{B-2.1}),	 then 
\begin{align*}
	\omega \ \text{is an abelian cocycle} &\iff
	a_{rst}=0
\	\text{for all} \ 1\leq r<s<t\leq n \\ &\iff \text{Every Yetter-Drinfeld module over}\ (kG,\omega) \ \text{is of diagonal type}.
\end{align*}

\end{lemma}\par 
Abelian cocycles possess an additional remarkable property: they can be resolved within an extended abelian group. Consider $\mathbb{G} \cong Z_{m_1^2} \times Z_{m_2^2} \times \cdots \times Z_{m_n^2}=\left\langle \sg_1\right\rangle \times \left\langle \sg_2\right\rangle \times \cdots \times \left\langle \sg_n\right\rangle $. There is a group surjection 
$\pi: \ \mathbb{G} \longrightarrow G, \ \ \sg_i\mapsto g_i$,  for $1 \leq i \leq n$.
We may pullback the $3$-cocycles over $G$ and obtain a $3$-cocycles over $\mathbb{G}$. That is, the map 
$$ \pi^*(\omega): \mathbb{G}\times \mathbb{G}\times \mathbb{G} \longrightarrow \mathbbm{k}^{\times}, \ \ (g,h,k)\mapsto \omega(\pi(g),\pi(h),\pi(k)).$$ is a $3$-cocycle over  $\mathbb{G}$. Actually, $\pi^*(\omega)$ is a coboundary. If we define $J_{\underline{a}}$ via $$J_{\underline{a}} :\mathbb{G} \times \mathbb{G} \longrightarrow \mathbbm{k}^{\times}, \ \	(\sg_1^{x_1}\cdots \sg_n^{x_n},\sg_1^{y_1}\cdots \sg_n^{y_n}) \mapsto \prod_{l=1}^{n}\zeta_{m_l^2}^{a_lx_l(y_l-y'_l)}\prod_{1\leq s <t\leq n }\zeta_{m_sm_t}^{a_{st}x_t(y_s-y'_s)},
$$then $\pi^*(\omega_{\underline{a}})=\partial(J_{\underline{a}})$, This result corresponds precisely to  \cite{QQG}.\par 
The following theorem classifies  finite-dimensional coradically graded pointed coquasi-Hopf algebras over  ﬁnite abelian groups.
\begin{thm} \textup{[\citealp{huang2024classification}, Theorem 5.2]}\label{B-thm 2.3}
 Let $M$ be a ﬁnite-dimensional coradically graded pointed coquasi-Hopf algebra over a ﬁnite abelian group with the coradical $M_0 = (\mathbbm{k}G, \omega)$. Then we have $M \cong \mathcal{B}(V)\# \mathbbm{k}G$ for a Yetter-Drinfeld module of ﬁnite type $V \in \GG$.
\end{thm}

\begin{rmk}
	\rm  
 We elaborate on Theorem \ref{B-thm 2.3} and present related results. Recall that if $V \in \GG$, then we can define the  support group $G_V:=<g \in G\mid ^gV\neq 0>$, which is a subgroup of $G$.\par 
	(i) A Yetter-Drinfeld module $\GG$ is said to be of finite type if \par 
$\circ$ $V$ has a standard basis. This is  equivalent to the condition that $\omega\mid_{G_V}$ is an abelian cocycle. \par 
$\circ$ $\Delta_{\chi,E}$ is an arithmetic root system and $\operatorname{ht}(\alpha) < \infty$ for all $\alpha \in \Delta_{\chi,E}$, where
$(\chi,E)$ is a pair associated to $V$.\par 
	(ii) By [\citealp{huang2024classification}, Theorem 3.9],  $\mathcal{B}(V)$ is finite-dimensional if and only if $V$ is finite type. Furthermore, in this case, $\mathcal{B}(V)$ is isomorphic to a Nichols algebra of diagonal type in  $\GGV$. Conversely, if   $\omega \mid_{G_V}$ is nonabelian, then $\mathcal{B}(V)$  is infinite-dimensional by [\citealp{huang2024classification}, Theorem 3.13].
\par (iii) 
The classification divides into two cases: \par
$\circ$ Finite-dimensional twisted Nichols algebras of diagonal type over 
$(G,\omega)$ with 
$\omega$ abelian;\par 
$\circ$ Finite-dimensional twisted Nichols algebras of non-diagonal type where $\omega\mid_{G_V}$
  is abelian.
\end{rmk}
\subsection{De-equivariantization of finite tensor categories}
In this subsection, we review the definition of de-equivariantization and present key results related to this construction.\par 
\begin{definition}
	$\operatorname{(i)}$ Let $\mathcal{C}$ be a finite braided tensor category. Suppose there is a braided tensor functor $\widetilde{i}:\mathcal{C} \rightarrow Z(\mathcal{D})$, such that the functor $Q : \mathcal{C}\overset{\widetilde{i}}{\rightarrow} Z(\mathcal{D}) \overset{F}{\rightarrow}\mathcal{D}$ is  a fully faithful tensor functor, where $F$ is the forget functor. Then we say $i$ is a central embedding.\par 
	$\operatorname{(ii)}$ For any central embedding $\operatorname{Rep}(G) \rightarrow \mathcal{D}$, we can deﬁne the de-equivariantization $\mathcal{D}_G$ , which is the tensor category of $\mathcal{O}(G)$-modules in $\mathcal{D}$, where $\mathcal{O}(G)$ is the linear dual of the group algebra.
\end{definition}
 Let $G$ be a finite group. There exists a profound relationship between braided central Hopf subalgebras of a Hopf algebra $H$ and central embedding  $\operatorname{Rep}(G) \rightarrow \HH$. \par 
\begin{definition}
	Let $H$ be a finite-dimensional Hopf algebra. A braided central Hopf subalgebra of $H$ is a pair $(K, r)$, where $K \subset H$ is a Hopf subalgebra, and $r: H \otimes K \rightarrow k$ is a bilinear form such that:
	\begin{align}
		r\left(h h^{\prime}, k\right) & =r\left(h^{\prime}, k_1\right) r\left(h, k_2\right),\label{B-2.3} \\
		r\left(h, k k^{\prime}\right) & =r\left(h_1, k\right) r\left(h_2, k^{\prime}\right), \label{B-2.4} \\
		r(h, 1) & =\varepsilon(h), \quad r(1, k)=\varepsilon(k), \label{B-2.5}\\
		r\left(h_1, k_1\right) k_2 h_2 & =h_1 k_1 r\left(h_2, k_2\right),\label{B-2.6} \\
		r\left(k, k^{\prime}\right) & =\varepsilon\left(k k^{\prime}\right),\label{B-2.7}
	\end{align}
	
	for all $k, k' \in K$, $h, h' \in H$.
\end{definition}
\begin{lemma}\textup{[\citealp{deequ}, Theorem 3.4]}\label{B-lem2.10}
	Let $H$ be a Hopf algebra and $K\subset H$ a commutative Hopf subalgebra. Then the following set of data are equivalent:\par 
	$\operatorname{(1)}$ A map $r : H \otimes K \rightarrow k$ such that $(K, r)$ is a braided central Hopf subalgebra of $H$.\par 
	$\operatorname{(2)}$ A braided tensor functor $F : {^K\mathcal{M}} \rightarrow \mathcal{Z}( {^H\mathcal{M}})= \HH$ such that the composition with the forgetful functor $\mathcal{Z}( {^H\mathcal{M}}) = \HH \rightarrow {^H\mathcal{M}}$ is fully faithful.
\end{lemma}
\subsection{Cohomology of finite tensor categories}
In this section, $\mathcal{C}$ and $\mathcal{D}$ are finite tensor categories. We present essential lemmas regarding their cohomology for subsequent applications. \par 
\begin{lemma}\textup{[\citealp{cohonichols}, Theorem 1.2.1]}\label{B-lem2.12}
Let $H$ be a ﬁnite-dimensional pointed Hopf algebra whose group of group-like elements is abelian. Then $\operatorname{Comod}(H)$ as well as $\operatorname{Rep}(H)$ satisfies \rm{\textbf{FGC}}.
\end{lemma}
\begin{lemma}\textup{[\citealp{coho}, Proposition 3.3]} \label{B-lem2.105}
	If $F: \mathcal{D}\rightarrow \mathcal{C}$ is a surjective tensor functor, and $\mathcal{D}$ satisfies \rm{\textbf{FGC}}, then $\mathcal{C}$ satisfies \rm{\textbf{FGC}}.
\end{lemma}
 \begin{lemma}\textup{[\citealp{coho}, Theorem 4.4]}\label{B-lem2.11}
 	Suppose $F: \mathcal{D}\rightarrow \mathcal{C}$ is a de-equivariantization of $\mathcal{D}$ with respect to a central embedding $\operatorname{Rep}(G)\rightarrow \mathcal{D}$, or equivalently an equivariantization of $\mathcal{C}$ with respect to a $G$-action. Then $\mathcal{D}$ satisfies \rm{\textbf{FGC}}  if and only if $\mathcal{C}$ satisfies \rm{\textbf{FGC}}.
  
 \end{lemma}
The authors of \cite{exactetingof} generalized Bruguieres and Natale's definition of exact sequences of tensor categories, introducing a new notion of exact sequences with respect to module categories. For brevity, we omit the explicit definition here.
\begin{lemma}\textup{[\citealp{coho}, Corollary 8.13]}\label{B-lem2.13}
Let $$ \mathcal{B}\longrightarrow \mathcal{C} 
\longrightarrow\mathcal{D} \boxtimes \operatorname{End}(\mathcal{M})$$ be an exact sequence of finite tensor categories. In particular, we assume $\mathcal{D}$ is a fusion category, then
	there is a natural identiﬁcation of $\mathbbm{k}$-algebras $$\operatorname{H}^{\bullet} (\mathcal{B}, \mathrm{1}) =\operatorname{H}^{\bullet} (\mathcal{C} , \mathrm{1}).$$ Furthermore, for any object $V$ in $\mathcal{C}$ , there is an identiﬁcation of $\operatorname{H}^{\bullet} (\mathcal{C} , \mathrm{1})$-modules, $\operatorname{H}^{\bullet} (\mathcal{B}, \mathscr{H}^0_{\mathcal{C}}(\mathcal{D}, V )) = \operatorname{H}^{\bullet} (\mathcal{C} , V)$, where the definition of $\mathscr{H}^0_{\mathcal{C}}$ is given in [\citealp{coho}, Definition 8.10].
\end{lemma}
\section{On finite-dimensional coquasi-Hopf algebras of diagonal type}
In this section, given a finite-dimensional coradically graded coquasi-Hopf algebra of diagonal type, we prove that $\operatorname{Comod}(M)$ satisfies \rm{\textbf{FGC}}. We begin by recalling key classification results. \par
\subsection{Overview of classification results} 
In this section, let $G \cong Z_{m_1} \times Z_{m_2} \times \cdots \times Z_{m_n}=\left\langle g_1\right\rangle \times \left\langle g_2\right\rangle \times \cdots \times \left\langle {g_n}\right\rangle $ with an abelian $3$-cocycle $\omega_{\underline{a}}$ and  $\mathbb{G} \cong Z_{m_1^2} \times Z_{m_2^2} \times \cdots \times Z_{m_n^2}=\left\langle \sg_1\right\rangle \times \left\langle \sg_2\right\rangle \times \cdots \times \left\langle \sg_n\right\rangle $.
We assume $\mathcal{B}(V)\in \GGAA$ is a finite-dimensional Nichols algebra of diagonal type. Without loss of generality, we can assume $G=G_V$ (the support group of $V$). As detailed in [\citealp{QQG} Section 4], the authors provided a method to classify such twisted Nichols algebras.  Here, we adopt an equivalent approach to describe the 'return trip' transformation between twisted and ordinary Nichols algebras for later use. 
\begin{center}
	\begin{tikzpicture}
 \node(G) at (-2,-4){Figure I};\node(H) at (6,-4){Figure II};
		\node (A) at (-2,0) {Given $\mathcal{B}(V)\in \GGAA$};
		\node (B) at (-2,-1.5) {$\mathcal{B}(V)\cong \mathcal{B}(V_{\mathcal{D}})\in \BGGA$};
		\node (C) at (-2,-3) {$\mathcal{B}(V_{\mathcal{D}})$ is twisted equivalent to  $\mathcal{B}(V_{\mathcal{D}})^{J_{\underline{a}}^{-1}} \in \BGGo$};
		\node (D) at (6,-3) { Given Nichols algebra $\mathcal{B}(V') \in \BGGo$};
		\node (E) at (6,-1.5) {$\mathcal{B}(V')^{J_{\underline{a}}}\in \BGGJ $};
		\node (F) at (6,0) {$\mathcal{B}(V')^{J_{\underline{a}}}\cong \mathcal{B}(V)\in \GGAA$ };
		\draw[->] (A) --node [right ] {} (B);
		\draw[->] (B) --node [ right] {} (C);	
		\draw[->] (D) --node [right ] {} (E);	
		\draw[->] (E) --node [right ] {If [\citealp{QQG}, $ (4.8)$] holds} (F);
	\end{tikzpicture}
\end{center}\par
Let us illustrate the above figures explicitly via  the language of  root data. Let  $\mathcal{D}=(\mathcal{D}_{\chi,E},S,X)$ be a root data, where: \par 
$\circ$ $\mathcal{D}_{\chi,E}$ is a Dynkin diagram of an arithmetic root system $\Delta_{\chi,E}$; \par $\circ$  $S=(s_{ij})_{m \times n}$, $X=(x_{ij})_{n \times m}$ are two matrices satisfying certain conditions (see [\citealp{QQG}, Definition 4.14]). \par
By [\citealp{QQG}, Corollary 4.13], 
for each root data $\mathcal{D}$, there exists unique  $\underline{a}\in \mathscr{A}'$ such that [\citealp{QQG}, (4.8)] holds for all $1\leq i\leq n$ and $1\leq j \leq m$.
\begin{equation}\label{B-2.2}
	\zeta_{m_i^2}^{m_i x_{i j}}=\zeta_{m_i^2}^{a_i s_{j i} m_i} \prod_{i<k \leq n} \zeta_{m_i m_k}^{a_{i k} s_{j k} m_i}
\end{equation}
For each root data $\mathcal{D}$, we can define a Nichols algebra $\mathcal{B}(V_{\mathcal{D}}) \in \BGGA$ as follows. Let $V_{\mathcal{D}}$ be the Yetter–Drinfeld module in $\BGGA$ with a canonical basis $\left\lbrace X_i \mid 1\leq i\leq m\right\rbrace $ such that
$$\delta_L\left(X_i\right)=\prod_{k=1}^n \sg_k^{s_{i k}} \otimes X_i, \quad 
\sg_i \rhd X_j=\zeta_{m_i^2}^{x_{i j}} \frac{J_{\underline{a}}\left(g_i, \prod_{k=1}^n \sg_k^{s_{i k}}\right)}{J_{\underline{a}}\left(\prod_{k=1}^n \sg_k^{s_{i k}}, g_i\right)} X_j.$$
\par The second statement of  [\citealp{QQG}, Theorem 4.15] tells us that each twisted Nichols algebra in $\GGAA$  must be isomorphic to  $\mathcal{B}(V_{\mathcal{D}}) \in \BGGA$ for some root data $\mathcal{D}$. This is  exactly what Figure I says. \par The first part of that theorem shows that each Nichols algebra in $\BGGA$ satisfying equation $(\ref{B-2.2})$ must be isomorphic to  $\mathcal{B}(V_{\mathcal{D}}) \in \BGGA$ for some root data $\mathcal{D}$. Moreover,  $\mathcal{B}(V_{\mathcal{D}})$ will be uniquely isomorphic to a twisted Nichols algebra in $\GGAA$. This corresponds to the Figure II. 
\par This completes the classification of finite-dimensional twisted Nichols algebras of diagonal type, and consequently, all finite-dimensional coquasi-Hopf algebras of diagonal type.
\subsection{Proof of diagonal type case}
For each $V' \in \BGGA$,
one may define a finite-dimensional Nichols algebra $\mathcal{B}(V') \in \BGGA$ and thus a finite-dimensional coquasi-Hopf algebra $H^{J_{\underline{a}}}=\mathcal{B}(V')\# \mathbbm{k}\mathbb{G}$. Through the processes of twisting and bosonization, we obtain a finite-dimensional coradically graded pointed Hopf algebra: $$H = \mathcal{B}(V')^{J^{-1}_{\underline{a}}} \# \mathbbm{k}\mathbb{G}.$$ On the other hand, if $\underline{a}$ satisfies equation (\ref{B-2.2}), we can define a finite-dimensional coquasi-Hopf algebra $$M\cong \mathcal{B}(V) \# \mathbbm{k}G,$$ where $\mathcal{B}(V)$ is isomorphic to $ \mathcal{B}(V')$. Note that
$\mathbb{G}$ contains a subgroup $\widetilde{\bg}=\left\langle \sg_1^{m_1}\right\rangle \times \left\langle \sg_2^{m_2}\right\rangle \times \cdots \times \left\langle \sg_n^{m_n}\right\rangle\cong G$.
We now state our main result:
\begin{prop}\label{B-thm3.1}
	With the above notions, if equation (\ref{B-2.2}) holds for all $1\leq i\leq n$ and $1\leq j \leq m$, then we have such a tensor equivalence 
	\begin{equation}
		\operatorname{Comod}(H)_{\widetilde{\bg}}\cong \operatorname{Comod}(M).
	\end{equation}
\end{prop}
\begin{proof}
Let $\chi_i$ be the generator of $\widehat{Z_{m_i^2}}$ such that $\chi_i(\sg_i)=\zeta_{m^2_i}$ for all $1 \leq i \leq n$. Define $\phi:\widetilde{\bg} \rightarrow \widehat{\bg}$ by:
\begin{equation}
    \phi((\sg_1^{m_1})^{y_1}\cdots (\sg_n^{m_n})^{y_n})=\prod_{l=1}^{n}\chi_l^{a_ly_lm_l}\prod_{1 \leq j<k \leq n}\chi_k^{a_{jk}y_jm_k}.
\end{equation} Let $k=(\sg_1^{m_1})^{y_1}\cdots (\sg_n^{m_n})^{y_n} \in \widetilde{\bg}$.
Then for any $k' \in \widetilde{\bg}$ and each generator $X_j$, $1\leq j \leq m $, we have\begin{align*}
    &< k', \phi(k)>=1, \\
    &k \rhd X_j= \prod_{i=1}^n \zeta_{m_i^2}^{m_iy_ix_{ij}}X_j\overset{(\ref{B-2.2})}{=}\prod_{i=1}^n\zeta_{m_i^2}^{a_iy_is_{ji}m_i}  \prod_{1\leq i< l \leq n}   \zeta_{m_im_l}^{a_{il}s_{jl}y_im_i}X_j= <\prod_{l=1}^n \sg_l^{s_{jl}},\phi(k)>X_j.
\end{align*}
Hence, the pair $(\widetilde{\bg},\phi)$ satisfies the condition (b)
in \textup{[\citealp{deequ}, Proposition 4.3]}, yielding a braided central Hopf subalgebra $(\mathbbm{k}\widetilde{\bg},r)$ of  $H$.
As $\widetilde{\bg}$ is abelian, we have $\operatorname{Rep}({\widetilde{\bg}})\cong \operatorname{Comod}(\mathbbm{k}\widetilde{\bg})$ as fusion category. By Lemma \ref{B-lem2.10}, there exists a braided tensor functor $F : \operatorname{Rep}(\widetilde{\bg})\rightarrow \mathcal{Z}( {^H\mathcal{M}})= \HH$ whose composition with the forgetful functor $\mathcal{Z}( {^H\mathcal{M}}) = \HH \rightarrow {^H\mathcal{M}}$ is fully faithful. This guarantees the existence of the de-equivariantization $\operatorname{Comod}(H)$, which defined as   the tensor category of  left $\mathbbm{k}\widetilde{\bg}$-modules in $\operatorname{Comod}(H)$.  \par 
On the other hand, there is an epimorphism of coquasi-Hopf algebras:
$$\pi:=(f\otimes p): \ H^{J_{\underline{a}}}=\mathcal{B}(V_{\mathcal{D}})\# \mathbbm{k}\mathbb{G} \longrightarrow M= \mathcal{B}(V) \# \mathbbm{k}G.$$ where $p:\mathbbm{k}\mathbb{G}\rightarrow\mathbbm{k}G$ is given by $\sg_i\mapsto g_i$ is surjective and $f: \mathcal{B}(V_{\mathcal{D}})\rightarrow \mathcal{B}(V)$ is an isomorphism. Direct computation shows:
$$\mathbbm{k}\widetilde{G}=(H^{J_{\underline{a}}})^{\operatorname{co}\pi}:=\left\lbrace x \in H^{J_{\underline{a}}}\mid (\operatorname{id}\otimes \pi)(\Delta(b))=b\otimes 1\right\rbrace.
$$ It follows by [\citealp{exactseq}, Example 6.4] that $\mathbbm{k}\widetilde{\bg}$ admits a structure of commutative algebra in $\mathcal{Z}(\operatorname{Comod}(H^{J_{\underline{a}}}))$, making the tensor category of left $\mathbbm{k}\widetilde{\bg}$-modules in $\operatorname{Comod}(H^{J_{\underline{a}}})$ is tensor equivalent to $\operatorname{Comod}(M)$. Note that, $\operatorname{Comod}(H)\cong \operatorname{Comod}(H^{J_{\underline{a}}})$ as tensor categories since $H^{J_{\underline{a}}}$ and $H$ differs by a cocycle-deformation. Hence the tensor category of left $\mathbbm{k}\widetilde{\bg}$-modules in $\operatorname{Comod}(H)$ is tensor equivalent to $\operatorname{Comod}(M)$.  We conclude: $\operatorname{Comod}(H)_{\widetilde{\bg}}\cong \operatorname{Comod}(M)$.
\end{proof}
\begin{thm}\label{B-thm3.2}
	For each finite-dimensional coradically graded coquasi-Hopf algebra $M$ of diagonal type,   $\operatorname{Comod}(M)$ satisfies \rm{\textbf{FGC}}.
\end{thm}
\begin{proof}
Given a finite-dimensional coradically graded pointed coquasi-Hopf algebra $M\cong \mathcal{B}(V) \# \mathbbm{k}G$ of diagonal type, Figure I associates it to  a unique Hopf algebra $H\cong \mathcal{B}(V_{\mathcal{D}})^{J_{\underline{a}}^{-1}}\# \mathbbm{k}\mathbb{G}$.  The equation (\ref{B-2.2}) holds automatically for $B(V_{\mathcal{D}})$ by [\citealp{QQG}, Theorem 4.15]. Then by Proposition \ref{B-thm3.1}, there is a subgroup $\widetilde{\bg}\subseteq \bg$ such that $	\operatorname{Comod}(H)_{\widetilde{\bg}}\cong \operatorname{Comod}(M)$. Since $\operatorname{Comod}(H)$ satisfies \rm{\textbf{FGC}} by Lemma \ref{B-lem2.105}, so does  $\operatorname{Comod}(M)$ by Lemma \ref{B-lem2.11}.
\end{proof}
\begin{rmk} \rm 
The referee  pointed out that \textup{[\citealp{deequ}, Proposition 4.3]} can simplify the proof of Proposition \ref{B-thm3.1}. We acknowledge referee for providing this better approach.
\end{rmk}
\section{Non-diagonal type case and proof of Theorem \ref{B-thm1.3}}

The classification result of ﬁnite-dimensional twisted Nichols algebras of non-diagonal type is much simpler, but the procedure is more difficult, see \cite{huang2024classification}. We first characterize all finite-dimensional simple twisted Nichols algebras of non-diagonal type over abelian groups.
\begin{lemma}\textup{[\citealp{huang2024classification}, Proposition 2.12]}
	Suppose $V \in \GG$ is a simple twisted Yetter-Drinfeld module of non-diagonal type, $\operatorname{deg}(V)=g$. Then $\mathcal{B}(V)$ is ﬁnite-dimensional if and only if V is one of the following two cases:\par 
	\rm	(C1) $g\rhd v=-v$ for all $v \in V,$\par 
	(C2) $\operatorname{dim}(V)=2$ and $g \rhd v=\zeta_3v$ for all $v \in V$.
\end{lemma}
\begin{rmk} 
    \rm In \cite{huang2024classification}, the authors  showed the following results. \par (1) For a rank-2 Yetter-Drinfeld module  $V=V_1 \oplus V_2$ with  $V_1,V_2$ are simple but nonisomorphic, $\mathcal{B}(V)$ is finite-dimensional if and only if $\operatorname{dim}(V_1)=\operatorname{dim}(V_2)=2$ and $V_1, V_2$ are of type (C1) with some additional restriction. \par 
    (2) Any  $V$ of rank $3$ with pairwise nonisomorphic simple twisted Yetter-Drinfeld modules yields an infinite-dimensional Nichols algebra $\mathcal{B}(V)$.
\end{rmk}
For brevity, we focus on simple twisted Yetter-Drinfeld modules of non-diagonal type, as our methods extend directly to rank $2$.  Consider the group $$G \cong \left\langle g_1,g_2,...,g_n\right\rangle$$ for $n \geq 3$, since $\omega$ is a non-abelian $3$-cocycle.  Let $V\in \GG$ be of either type (C1) or (C2) with $\operatorname{deg}(V)=g_1$ and $\operatorname{dim}(V)=m \geq2$.  
By [\citealp{huang2024classification}, Proposition 3.10, Theorem 3.9], there is a finite abelian group $$\mathbb{G} \cong \left\langle \mathbbm{g}_1\right\rangle  \times \left\langle \sg_2 \right\rangle \times \cdots \times \left\langle \sg_n\right\rangle $$ with generators $\sg_i$, such that $\operatorname{ord}(\sg_i)=\operatorname{ord}(g_i)^2$.
Moreover, there exists a group epimorphism $\pi:\mathbb{G}\longrightarrow G$ satisfying $\pi(\sg_i)=g_i$ and a section $\iota$ to $\pi$ such that $\iota(g_i)=\sg_i$ for all $1\leq i \leq n$. Furthermore,  there exists $\widetilde{V} \in \GGB$ such that $\mathcal{B}(\widetilde{V})\cong \mathcal{B}(V)$,
$\pi^*(\omega) \in \mathscr{A}''$ and the support group $\mathbb{G}_{\widetilde{V}}$ of  $\widetilde{V}$ is $ \left\langle \sg_1 \right\rangle$. The following lemma is  immediate. 
\begin{lemma}\label{B-lem4.2}
	The bosonization $H:=\mathcal{B}(\widetilde{V}) \# \mathbbm{k}\bg_{\widetilde{V}}$ is a Hopf algebra and  satisfies \rm{\textbf{FGC}}.
\end{lemma}

\begin{proof}
	Since $\pi^*{\omega} \in \mathscr{A}''$, the restriction $\pi^*{\omega}\mid_{\bg_{\widetilde{V}}\times \bg_{\widetilde{V}} \times \bg_{\widetilde{V}}} \equiv 1$. By [\citealp{huang2024classification}, Corollary 3.16], $\mathcal{B}(\widetilde{V})$ is isomorphic to a Nichols algebra of diagonal type in $\GGVV$(retaining the notation $\mathcal{B}(\widetilde{V})$ for simplicity). Thus $H:=\mathcal{B}(\widetilde{V}) \# \mathbbm{k}\bg_{\widetilde{V}}$ is a Hopf algebra, with an abelian group of group-like elements. Hence $H$ satisfies \rm{\textbf{FGC}} by Lemma \ref{B-lem2.12}.
\end{proof}
Let $M:= \mathcal{B}(V)\# \mathbbm{k}G$ and  $\widetilde{M}:=\mathcal{B}(\widetilde{V})\# \mathbbm{k}\mathbb{G}$. To prove $\operatorname{Comod}(M)$ satisfies \rm{\textbf{FGC}}, we are going to show $\operatorname{Comod}(\widetilde{M})$ does, leveraging the surjective coquasi-Hopf algebra map $\widetilde{M}\rightarrow M$.
\begin{thm}\label{B-thm4.3}
	Both $\operatorname{Comod}(\widetilde{M})$ and $\operatorname{Comod}(M)$ satisfy \rm{\textbf{FGC}}.
\end{thm}
\begin{proof}
	 Clearly, $\widetilde{M}$ and $H$ are generated by group-like and skew primitive elements. Since $\operatorname{dim}(V)=m$ and $\mathcal{B}(V)\cong \mathcal{B}(\widetilde{V})$. We may assume $\widetilde{V}=\operatorname{span}\left\lbrace x_1,x_2,...,x_m\right\rbrace $. Thus, $\widetilde{M}$ is generated by $$\left\lbrace x_1,x_2,...,x_m, \sg_j^i\mid 0\leq i <m_j^2, \ 1\leq j \leq n\right\rbrace, $$ 
	while $H$ is generated by 
	$${	\left\lbrace x_1,x_2,...,x_m, \sg_1^i \mid 0\leq i <m_1^2\right\rbrace}.$$ \par  
The injective coquasi-Hopf algebra map
$i: H \rightarrow \widetilde{M},\  x_j\# \sg_1^i\mapsto x_j \#\sg_1^i$ for $0\leq i \leq n-1$, $1 \leq j \leq m,$
induces a fully faithful tensor functor
	\begin{equation}
	    \iota:  \operatorname{Comod}(H) \longrightarrow \operatorname{Comod}(\widetilde{M}).
	\end{equation}
\par 
	Note that $p: \mathbb{G}\rightarrow \bg/{\bg_{\widetilde{V}}}$ is a group epimorphism given by $p(\sg_1)=1$ and $p(\sg_i)=\sg_i$ for all $2\leq i \leq n$ and $\pi: \mathcal{B}(V) \rightarrow \mathbbm{k}, 1_{\mathcal{\mathcal{B}(V)}}  \mapsto 1_{\mathcal{B}(V)}, \ \ x_i \mapsto 0 $ is an algebra map and a coalgebra map. This yields a coquasi-Hopf algebra surjection:
	$$f=(\pi\otimes p):\mathcal{B}(\widetilde{V}) \# \mathbbm{k}\bg \longrightarrow  \mathbbm{k}{\bg/{\bg_{\widetilde{V}}}}, \ x \#\sg \mapsto \pi(x)\#p(\sg).$$
				which induces a surjective tensor functor:
	\begin{equation} F: \operatorname{Comod}(\widetilde{M}) \rightarrow \operatorname{Vec}_{\bg/{\bg_{\widetilde{V}}}}^{\omega _{\bg/{\bg_{\widetilde{V}}}}}.
	\end{equation}
We denote the fusion category $\operatorname{Vec}_{\bg/{\bg_{\widetilde{V}}}}^{\omega _{\bg/{\bg_{\widetilde{V}}}}}$ as $\mathcal{D}$ for simplicity. Since $f\circ i(x\#\sg_1^i)=\pi(x)\#1 \in \mathbbm{k}$, we see that $F \circ \iota(X)\in \operatorname{Vec}$ for all $X \in \operatorname{Comod}(H)$ by definition of $\iota$ and $F$. That is $\operatorname{Comod}(H) \subseteq\operatorname{Ker}(F)$. \par 
All categories involved are comodule categories of finite-dimensional coquasi-Hopf algebras. Hence $\operatorname{FPdim}(\operatorname{Comod}(H))=\operatorname{dim}_{\mathbbm{k}}(H)$, $\operatorname{FPdim}(\operatorname{Comod}(\widetilde{M}))=\operatorname{dim}_{\mathbbm{k}}(\widetilde{M})$ and $\operatorname{FPdim}(\mathcal{D})=\operatorname{dim}_{\mathbbm{k}}\mathbbm{k}(\bg/{\bg_{\widetilde{V}}})$ by \textup{[\citealp{tensor}, Example 6.1.9]}.  Clearly, \begin{equation*}
\operatorname{dim}_{\mathbbm{k}}(\widetilde{M})=\operatorname{dim}_{\mathbbm{k}}(H)\cdot\operatorname{dim}_{\mathbbm{k}}\mathbbm{k}(\bg/{\bg_{\widetilde{V}}}).
\end{equation*}
Hence 
\begin{equation}
	\operatorname{FPdim}(\operatorname{Comod}(\widetilde{M}))=\operatorname{FPdim}(\operatorname{Comod}(H))\cdot\operatorname{FPdim}(\mathcal{D}).
\end{equation}
By  [\citealp{exactetingof}, Theorem 3.4], there is an exact sequence of finite tensor category
\begin{equation}
	\operatorname{Comod}(H) \longrightarrow \operatorname{Comod}(\widetilde{M})\longrightarrow \mathcal{D}.
\end{equation}
 Lemma \ref{B-lem2.13} provides a natural identiﬁcation of $\mathbbm{k}$-algebras $$\operatorname{H}^{\bullet} (\operatorname{Comod}(H), \mathrm{1}) =\operatorname{H}^{\bullet} (\operatorname{Comod}(\widetilde{M}) , \mathrm{1}).$$ Since$\operatorname{Comod}(H)$ satisfies \rm{\textbf{FGC}}  by Lemma \ref{B-lem4.2}, $\operatorname{H}^{\bullet} (\operatorname{Comod}(\widetilde{M}) , \mathrm{1})$ is a finitely generated algebra. Moreover, for each $W \in \operatorname{Comod}(\widetilde{M})$, there is an identification of $\operatorname{H}^{\bullet} (\operatorname{Comod}(\widetilde{M}) , \mathrm{1})$ module via $$\operatorname{H}^{\bullet} (\operatorname{Comod}(H), \mathscr{H}_{\operatorname{Comod}(\widetilde{M})}^0(\mathcal{D},W)) =\operatorname{H}^{\bullet} (\operatorname{Comod}(\widetilde{M}) , W).$$
Since $\operatorname{H}^{\bullet} (\operatorname{Comod}(H), \mathscr{H}_{\operatorname{Comod}(\widetilde{M})}^0(\mathcal{D},W))  $ is  a finitely generated module of $\operatorname{H}^{\bullet}(\operatorname{Comod}(H),1)$, then $\operatorname{H}^{\bullet} (\operatorname{Comod}(\widetilde{M}) , W)$ is a finitely generated module of $\operatorname{H}^{\bullet} (\operatorname{Comod}(\widetilde{M}) , 1).$ Hence $\operatorname{Comod}(\widetilde{M})$ satisfies \rm{\textbf{FGC}}. Recall that there is a surjection of coquasi-Hopf algebra 
$\widetilde{M} \rightarrow M$, which induces a surjective tensor functor $\operatorname{Comod}(\widetilde{M}) \rightarrow \operatorname{Comod}(M)$. Thus $\operatorname{Comod}(M)$ satisfies \rm{\textbf{FGC}} by lemma \ref{B-lem2.105}.
\end{proof}

\begin{proof}[\rm{\textbf{Proof of Theorem 1.3}}]
	Since $\mathcal{C} \cong \operatorname{Comod}(M)$ for some finite-dimensional coradically graded coquasi-Hopf algebra $M$ over an abelian group, the result follows directly from Theorem \ref{B-thm3.2} and \ref{B-thm4.3}.
\end{proof}
\vspace{3mm}
\noindent\textbf{Acknowledgment} This work is supported by the  NSFC 12271243 and National Key R\&D Program of China 2024YFA1013802. Furthermore, we would like thank the referee for his/her very valuable comments which improved the paper greatly.

\bibliographystyle{plain}\small
\bibliography{f.g.-ref}
\end{document}